\documentclass[11pt]{article}

\usepackage{amssymb,amsmath,amsfonts,amsthm,bbm}
\usepackage[dvips]{graphicx}
\usepackage[ansinew]{inputenc}
\usepackage{enumerate}
\usepackage{multirow}
\usepackage{subcaption}
\usepackage[numbers]{natbib}

\theoremstyle{plain}
\newtheorem{theorem}{Theorem}[section]
\newtheorem{corollary}[theorem]{Corollary}
\newtheorem{lemma}[theorem]{Lemma}
\newtheorem{proposition}[theorem]{Proposition}
\newtheorem{problem}[theorem]{Problem}
\newtheorem{remark}[theorem]{Remark}

\newcommand{\R}{\mathbb{R}}

\newcommand{\Exp}{\operatornamewithlimits{Exp}}
\newcommand{\E}{\mathbb{E}}
\newcommand{\p}{\mathbb{P}}

\newcommand{\argmax}{\operatornamewithlimits{argmax}}

\begin{document}
\title{Expected Supremum Representation of a Class of Single Boundary Stopping Problems}

\author{Luis H. R. Alvarez E.\thanks{luis.alvarez@tse.fi}\quad Pekka Matomäki \thanks{pjsila@utu.fi}\\
Department of Accounting and Finance\\ Turku School of Economics\\
FIN-20014 University of Turku\\ Finland}
\maketitle

\abstract{We consider the representation of the value of a class of optimal stopping problems of linear diffusions in a linearized form as an expected supremum of a known function.
We establish an explicit integral representation of this representing function by utilizing the explicitly known marginals of the joint probability distribution of the extremal processes.
We also delineate circumstances under which the value of a stopping problem induces directly this representation and show how it is connected with the monotonicity of the generator.
We compare our findings with existing literature and show, for example, how our representation is linked to the smooth fit principle and how it coincides
with the optimal stopping signal representation.
The intricacies of the developed integral representation are explicitly illustrated in various examples arising in financial applications of optimal stopping.}\\



\thispagestyle{empty} \clearpage \setcounter{page}{1}

\section{Introduction}

It is well-known from the literature on stochastic processes that the probability distributions of first hitting times are closely related to the probability
distributions of the running supremum and running infimum of the underlying diffusion. Consequently, the question of whether a linear diffusion has exited
from an open interval prior to a given date or not can be answered by studying the behavior of the extremal processes up to the date in question. If the extremal
processes have remained in the open interval up to the particular date, then the process has not yet hit the boundaries and vice versa. In
this study we utilize this connection and develop a linearized representation of the value function of an optimal stopping problem as the expected supremum of a representing function
with known properties in the spirit of the pioneering work by \cite{FoKn1,FoKn2} and its subsequent extension to the treatment of optimal stopping problems by \cite{ChSaTa}.
More formally, we plan to determine {\em explicitly} the nondecreasing, nonnegative, and upper semicontinuous representing function $f$ for which
\begin{align}
V(x)= \E_x\left[\sup\{f\left(X_t\right);t\leq T\}\right],\label{BankFollmerrep}
\end{align}
where $V(x)$ denotes the value of the considered class of optimal stopping problems and $T\sim \text{Exp}(r)$ is an exponentially distributed random time independent of
the underlying process $X$.

The relatively recent literature on stochastic control theory indicates that the connection between, among others, the value functions and extremal processes in optimal stopping and
singular stochastic control problems goes far beyond the standard connection between first hitting times and the running supremum
and infimum of the underlying process (see, for example, \cite{BaBa,BaElKa,BaFo,BaRi,ChFe14,DeFeMo15,ElKaMe,ElKaFo,Ferrari15,FeSa14}). Essentially, in these studies the determination of the
optimal policy and its value is shown to be equivalent with the existence of an appropriate optional projection involving the running supremum of a progressively
measurable process (known as the {\em Bank - El Karoui representation}). The advantage of the representation utilized in these studies is that it is very general and applies
also outside the standard Markovian and infinite horizon setting.
Moreover, it can be utilized for studying and solving other stochastic control problems as well. For example, as was shown in \cite{BaElKa,BaFo},
the approach is applicable in the analysis of the Gittins-index familiar from the literature on multi-armed bandits (cf. \cite{ElKarKar94,Git79,GitGla77,GitJon79,Kar1984}).

Instead of establishing directly how the value of the considered class of optimal stopping problems can be expressed as an expected supremum, we take an alternative route and compute
first explicitly the expected value of the supremum of an unknown function satisfying a set of
monotonicity and regularity conditions by utilizing the known probability distribution of the running supremum of the underlying. Setting this expected value equal with the value of the optimal stopping problem then results into a functional identity
from which the unknown function can be explicitly determined. In the considered single boundary setting the function admits a relatively simple characterization in terms of
the increasing minimal excessive mapping for the underlying diffusion (cf. \cite{BaBa}). We find that the required monotonicity of the function needed for the representation is closely
related with the monotonicity
of the generator on the state space of the underlying process. However, since only the sign of the generator typically affects the determination of the optimal strategy
and its value, our results
demonstrate that not all single boundary problems can be represented as the expected supremum of a monotonic function. We also investigate the regularity properties of the
function needed for the representation and show that
it needs not be continuous at the optimal stopping boundary. More precisely, we find that if the optimal boundary is attained at a point where the exercise payoff is not differentiable and, hence, the standard smooth fit condition is not satisfied, then the
representing function is only upper semicontinuous at the optimal boundary. This is a result which is in line with the findings by \cite{ChSaTa}.

The contents of this study is as follows. In section two we formulate the considered problem, characterize the underlying stochastic dynamics, and state a set of
auxiliary results needed in the subsequent analysis of the problem. Section three focuses on a single boundary setting where the optimal rule is to exercise as soon
as a given exercise threshold is exceeded.  Our general findings on the representing function are explicitly illustrated  in section four in various settings including incentive compatible stopping rules, Gittins indices, optimal entry, and stopping of spectrally negative jump diffusions. Finally, section five concludes our study.

\section{Problem Formulation}

\subsection{Underlying stochastic dynamics}
We consider a linear, time homogeneous and regular diffusion process $X=\{X(t);t\in[0,\xi)\}$, where $\xi$ denotes the possible infinite life time of the
diffusion. We assume that the diffusion is defined on a complete filtered probability space $(\Omega,\p,\{\mathcal{F}_t\}_{t\geq0},\mathcal{F})$, and that the state
space of the diffusion is $\mathcal{I}=(a,b)\subset\R$. Moreover, we assume that the diffusion does not die inside $\mathcal{I}$, implying that the boundaries $a$
and $b$ are either natural, entrance, exit or regular (see Section II. 1 in \cite{BorSal15} for a characterization of the boundary behaviour of diffusions). If a boundary
is regular, we assume that it is killing and that the process $X$ is immediately sent to a cemetery state $\partial\not\in \mathcal{I}$ as soon at it hits that boundary. Furthermore we will denote by $M_t=\sup\{X_s;s\in[0,t]\}$ the running supremum process of the considered diffusion $X_t$.

As usually, we denote by $\mathcal{A}$ the differential operator representing the infinitesimal generator of $X$. For a given smooth mapping $f:\mathcal{I}\mapsto\R$ this operator is given by
\begin{align*}
(\mathcal{A}f)(x)=\frac{1}{2}\sigma^2(x)\frac{d^2}{dx^2}f(x)+\mu(x)\frac{d}{dx}f(x),
\end{align*}
where the drift coefficient $\mu:\mathcal{I}\mapsto\R$ and the volatility coefficient $\sigma:\mathcal{I}\mapsto\R_+$ are given continuous mappings. In order to avoid interior singulairties, we assume throughout this study that $\sigma(x)>0$ for all $x\in \mathcal{I}$.  As is known from the classical
theory on linear diffusions,
there are two linearly independent {\it fundamental solutions}
$\psi(x)$ and $\varphi(x)$ satisfying a set of appropriate boundary
conditions based on the boundary behavior of the process $X$ and
spanning the set of solutions of the ordinary differential equation
$(\mathcal{G}_ru)(x)=0$, where $\mathcal{G}_r=\mathcal{A}-r$ denotes the differential operator associated with the diffusion $X$ killed at the constant rate $r$.
Moreover, $\psi'(x)\varphi(x) - \varphi'(x)\psi(x) = BS'(x),$
where $B>0$ denotes the constant Wronskian of the fundamental
solutions and
$$
S'(x)=\exp\left(-\int^x\frac{2\mu(t)}{\sigma^2(t)}dt\right)
$$
denotes the density of the scale function of $X$ (for a comprehensive characterization of the fundamental solutions, see \cite{BorSal15}, pp. 18--19). The
functions $\psi$ and $\varphi$ are minimal in the sense that any non-trivial $r$-excessive mapping for $X$ can be expressed as a combination of these two
(cf. \cite{BorSal15}, pp. 32--35). Given the fundamental solutions, let $u(x)=c_1\psi(x) + c_2\varphi(x), c_1,c_2\in \mathbb{R}$ be an arbitrary twice continuously
differentiable $r$-harmonic function and define for sufficiently smooth mappings $g:\mathcal{I}\mapsto\R$ the functional
\begin{align*}
(L_u g)(x) = g(x) \frac{u'(x)}{S'(x)}-\frac{g'(x)}{S'(x)}u(x) = c_1(L_\psi g)(x)+c_2(L_\varphi g)(x)
\end{align*}
associated with the representing measure for $r$-excessive functions (cf. \cite{Salminen1985}). Noticing that if $g$ is twice continuously differentiable, then
\begin{align}
(L_u g)'(x)=-(\mathcal{G}_rg)(x)u(x)m'(x)\label{linearityGenerator}
\end{align}
where
$
m'(x) = 2/(\sigma^{2}(x)S'(x))
$
denotes the density of the speed measure $m$ of $X$.
Hence, we find that
\begin{align}
(L_u g)(y)-(L_u g)(z)=\int_y^z(\mathcal{G}_rg)(v)u(v)m'(v)dv\label{canonical}
\end{align}
for any $a<y<z<b$.

Finally, we denote by $\mathcal{L}_r^1(\mathcal{I})$ the class of measurable functions $f:\mathcal{I}\mapsto\R_+$ satisfying the integrability condition
$$
\mathbb{E}_x\int_0^\infty e^{-rs}|f(X_s)|ds<\infty
$$
for all $x\in \mathcal{I}$. As is known from the literature on linear diffusions, if $f\in \mathcal{L}_r^1(\mathcal{I})$ then its expected cumulative present value
$$
(R_rf)(x)=\mathbb{E}_x\int_0^\infty e^{-rs}f(X_s)ds
$$
can be expressed as (cf. \cite{BorSal15}, p. 29)
\begin{align}
(R_rf)(x)= \int_a^bG_r(x,v)f(v)m'(v)dv,\label{Green}
\end{align}
where
\begin{align}
G_r(x,v)=\begin{cases}
B^{-1}\varphi(v)\psi(x),&x\leq v,\\
B^{-1}\varphi(x)\psi(v),&x\geq v.
\end{cases}
\end{align}

\subsection{The Optimal Stopping Problem and Auxiliary Results}

In this paper our objective is to examine the optimal stopping problem
\begin{align}\label{eq prob}
V(x)=\sup_{\tau}\E_x\left[e^{-r\tau}g(X_\tau)\right]
\end{align}
for exercise payoff functions $g$ satisfying a set of sufficient regularity conditions and establish a representation of the value $V$ as
the expected supremum of an appropriately chosen representing function along
the lines of the pioneering studies \cite{BaElKa}, \cite{BaFo}, \cite{ChSaTa}, \cite{ElKaFo}, \cite{ElKaMe}, \cite{FoKn1}, \cite{FoKn2}. Our main result is
based on the following representation theorem originally established in \cite{ChSaTa}.

\begin{theorem}(\cite{ChSaTa}, Theorem 2.5)\label{theorem abo1}
Let $X_t$ be a Hunt process on $\mathcal{I}$ and $T\sim\Exp (r)\perp X_t$. Assume that the exercise payoff $g$ is non-negative, lower semicontinuous, and satisfies the condition
$\E_x\left[\sup_{t\geq 0}e^{-rt}g(X_t)\right]<\infty$ for all $x\in \mathcal{I}$.
Assume also that there exists an upper semicontinuous $\hat{f}$ and a point $y^{\ast}\in \mathcal{I}$ such that
\begin{enumerate}
\item[(a)] $\hat{f}(x)\leq0$ for $x< y^{\ast}$, $\hat{f}(x)$ is non-decreasing and positive for $x\geq y^{\ast}$,

\item[(b)] $\E_x\left[\sup_{0\leq t\leq T}\hat{f}(X_t)\right]=g(x)$ for $x\geq y^{\ast}$, and
\item[(c)] $\E_x\left[\sup_{0\leq t\leq T}\hat{f}(X_t)\right]\geq g(x)$ for $x\leq y^{\ast}$.
\end{enumerate}
Then
\begin{align}\label{aboarvo1}
V(x)=\E_x\left[\sup_{0\leq t\leq T}\hat{f}(X_t)\mathbbm{1}_{[y^{\ast},b)}(X_t)\right]=\E_x\left[\hat{f}(M_T)\mathbbm{1}_{[y^{\ast},b)}(M_T)\right]
\end{align}
and  $\tau^*=\inf\{t\geq0: X_t>y^{\ast}\}$ is an optimal stopping time.
\end{theorem}

This theorem essentially states that if we can find a representing function $\hat{f}$ satisfying the required conditions (a)-(c), then the optimal stopping policy for \eqref{eq prob}
constitutes an one-sided threshold rule and its value can be expressed in a linearized form as an expected supremum attained at an independent exponential random time.
As we will prove later in this paper, the reverse argument is  also sometimes true: under certain circumstances the value of
the optimal policy generates a continuous and monotone function $\hat{f}$ for which the representation \eqref{aboarvo1} is valid. However, as we will
point out later in the case where the exercise reward can be expressed as an expected cumulative present value of a continuous flow, all single boundary stopping problems cannot be represented as proposed in Theorem \ref{theorem abo1}.

Before proceeding in our analysis and the explicit identification of the representing function, we first establish two auxiliary lemmata needed in the analysis of the problem. Our first findings based on the known joint probability distribution of the underlying and its running supremum are summarized in the following.
\begin{lemma}\label{jointprobabilitydist}
(A) If $h:\mathcal{I}\mapsto \mathbb{R}$ satisfies $h\in \mathcal{L}_r^1(\mathcal{I})$, then
\begin{align}\label{condexpb}
\frac{1}{r}\E_x[h(X_T)|M_T\leq y] &=\frac{(R_rh)(x)-(R_rh)(y)\frac{\psi(x)}{\psi(y)}}{1-\frac{\psi(x)}{\psi(y)}}
\end{align}
and
\begin{align}\label{condexp}
\frac{1}{r}\E_x[h(X_T)|M_T=y]=\frac{S'(y)}{\psi'(y)}\int_a^yh(v)\psi(v)m'(v)dv
\end{align}
for all $x\in(a,y]$. Especially, if $h\in C(\mathcal{I})\cap\mathcal{L}_r^1(\mathcal{I})$, then
\begin{align}\label{condexpc}
\begin{split}
\lim_{x\uparrow y}\frac{1}{r}\E_x[h(X_T)|M_T\leq y] &= (R_rh)(y)-(R_rh)'(y)\frac{\psi(y)}{\psi'(y)}\\
&=\frac{S'(y)}{\psi'(y)}\int_a^yh(v)\psi(v)m'(v)dv
\end{split}
\end{align}
for all $y\in \mathcal{I}$.\\
(B) If $\mathcal{G}_rh\in C(\mathcal{I}\backslash \mathcal{P})\cap\mathcal{L}_r^1(\mathcal{I})$, where $\mathcal{P}\in \mathcal{I}$ is a finite set of points, and
$\lim_{x\downarrow a}(L_\psi h)(x)=0$, then for all $x\in(a,y]$
\begin{align}\label{condexprep}
\frac{1}{r}\E_x[(\mathcal{G}_rh)(X_T)|M_T=y] &=-\frac{(L_\psi h)(y)}{(L_\psi \mathbbm{1})(y)},
\end{align}
where $\mathbbm{1}=\mathbbm{1}_{\mathcal{I}}(x)$.
\end{lemma}
\begin{proof}
See Appendix \ref{app2}.
\end{proof}
Second, in order to characterize how increased volatility affects the representing function, we need to state conditions under which the sign of the impact of increased volatility on the Laplace transform of the first hitting time to a constant boundary can be unambiguously described. A set of sufficient conditions under which this sign is positive are now stated in the following.
\begin{lemma}\label{increasedvolatility}
If $\psi(x)$ is convex, then increased volatility increases or leaves unchanged the ratio $\psi(x)/\psi(z)$ for all $a<x\leq z<b$ and decreases or leaves unchanged the ratio $\psi'(x)/\psi(x)$ for all $x\in \mathcal{I}$.
\end{lemma}
\begin{proof}
See Appendix \ref{app3}.
\end{proof}

\section{Representation as Expected Supremum}\label{sec 1sided}

\subsection{Problem Setting}

Our main objective is now to delineate general circumstances under which the value of a one-sided threshold policy can be expressed as the expected supremum of a monotonic representing function
and to identify that function explicitly. In what follows, we will focus on the case where the considered stopping policy can be characterized as a rule where the underlying
process is stopped as soon as it exceeds a given constant threshold. The case where the single boundary stopping rule is to exercise as soon as the underlying falls below a
given constant threshold is completely analogous and, therefore, left untreated.

Let $g:\mathcal{I}\mapsto\R$ be a continuous payoff function satisfying the condition $g^{-1}(\mathbb{R}_+)=(x_g,b)\neq\emptyset$ for some $x_g\in\mathcal{I}$ and
\begin{align}\label{integrability}
\mathbb{E}_x\left[\sup_{t\geq 0}e^{-rt}g(X_t)\right]<\infty
\end{align}
for all $x\in \mathcal{I}$. Assume also that $g\in C^1(\mathcal{I}\setminus \mathcal{P})\cap C^2(\mathcal{I}\setminus \mathcal{P})$, where
$\mathcal{P}\in \mathcal{I}$ is a finite set of points in $\mathcal{I}$ and that $|g'(x\pm)|<\infty$ and $|g''(x\pm)|<\infty$ for all $x\in \mathcal{P}$.

Given the assumed regularity conditions, let $\tau_y=\inf\{t\geq0: X_t\geq y\}$ denote the first exit time of the underlying diffusion from the set $(a, y)$,
where $y\in g^{-1}(\mathbb{R}_+)$. Define now the parameterized family of nonnegative and continuous functions $V_y:\mathcal{I}\mapsto \mathbb{R}_+$ by
\begin{align}\label{eq Vy}
V_y(x)=\E_x\left[e^{-r\tau_y}g(X_{\tau_y});\tau_y<\infty\right] =\begin{cases}
g(x)\quad &x\geq y\\
\psi(x)\frac{g(y)}{\psi(y)}\quad& x<y.
\end{cases}
\end{align}
Given representation \eqref{eq Vy}, we can now state our identification problem as follows.
\begin{problem}\label{1reuna}
(A) For a given $y\in g^{-1}(\mathbb{R}_+)$, does there exist a nonnegative function $\hat{f}:\mathcal{I}\mapsto \mathbb{R}_+$ such that for all $x\in\mathcal{I}$ we would have
\begin{align}
J_y(x):=\E_x\left[\hat{f}(M_T)\mathbbm{1}_{[y,b)}(M_T)\right]=V_{y}(x).\label{prob1}
\end{align}
(B) Under which conditions on the function $\hat{f}$ and the threshold $y$ we have
$$\hat{f}(M_T)\mathbbm{1}_{[y,b)}(M_T) =\sup_{t\in[0,T]}\{\hat{f}(X_t)\mathbbm{1}_{[y,b)}(X_t)\}$$
and, consequently,
\begin{align}
V(x)=V_y(x)=\E_x\left[\sup_{t\in[0,T]}\{\hat{f}(X_t)\mathbbm{1}_{[y,b)}(X_t)\}\right].\label{prob2}
\end{align}
\end{problem}
It's worth emphasizing  that Problem \ref{1reuna} is twofold. The first representation problem essentially asks if the
expected value of the exercise payoff accrued at the first hitting time to a constant boundary can be expressed as the expected value of an yet unknown representing function $\hat{f}$ at
the running maximum of the underlying diffusion at an independent exponentially distributed date. The second question essentially asks when the function $\hat{f}$ is such that the representation agrees with the general functional form utilized in Theorem \ref{theorem abo1} and in that way results into the value of the considered stopping problem. As we will later establish in this paper, the class of functions satisfying the first representation is strictly larger than the latter.

\subsection{Standard Sufficiency Conditions}
Before proceeding in the derivation of the representation as an expected supremum, we first need to characterize sufficient conditions under which the value $V_y$ coincides with the value of the optimal stopping problem \eqref{eq prob}. To accomplish this, we follow the Martin boundary representation approach introduced in the pioneering study \cite{Salminen1985} (for associated results focusing precisely on single-boundary problems, see \cite{CrMo14}) and
establish the following result characterizing the optimal policy. We apply this result later for the
identification of circumstances under which the value of the considered one-sided problem can be expressed as the expected supremum of a monotonic function.
\begin{lemma}\label{Martin1}
Assume that the following conditions are satisfied:
\begin{itemize}
  \item[(i)] there exists a $y^\ast = \argmax\{g(x)/\psi(x)\}\in \mathcal{I}$,
  \item[(ii)] $(\mathcal{G}_rg)(x)\leq 0$ for all $x\in [y^\ast,b)\setminus \mathcal{P}$
  \item[(iii)] $g'(x+)\leq g'(x-)$ for all $x\in[y^\ast,b)\cap \mathcal{P}$
\end{itemize}
Then $V(x)=V_{y^\ast}(x)$ and $\tau_{y^\ast}=\inf\{t\geq 0: X_t\geq y^\ast\}$ is an optimal stopping time.
\end{lemma}
\begin{proof}
See Appendix \ref{app4}.
\end{proof}

\begin{remark}\label{sufficient}
It is at this point worth emphasizing that under the following slightly stricter assumptions there always exists a unique maximizing threshold
$y^\ast = \argmax\{g(x)/\psi(x)\}$ and the conditions of Lemma \ref{Martin1} are satisfied
(cf. Lemma 3.4 in \cite{AlMaRa14}):
\begin{itemize}
\item[(A)] $g^{-1}(\mathbb{R}_-)=(a,y_0)$, where $a<y_0<b$, and $b$ is unattainable for $X$,
\item[(B)]  there exists a $\tilde{x}\in \mathcal{I}$ so that $(\mathcal{G}_rg)(x) \geq 0$ for all $x\in (a,\tilde{x})\setminus \mathcal{P}$ and
 $(\mathcal{G}_rg)(x)< 0$ for all $x\in (\tilde{x},b)\setminus \mathcal{P}$,
 \item[(C)] $g'(x+)\geq g'(x-)$ for all $x\in(a,\tilde{x})\cap \mathcal{P}$ and $g'(x+)\leq g'(x-)$ for all $x\in[\tilde{x},b)\cap \mathcal{P}$
\end{itemize}
\end{remark}
These assumptions of Remark \ref{sufficient} are typically met in financial applications of optimal stopping.
Note that these conditions do not impose monotonicity requirements on the behavior of the generator $\mathcal{G}_rg$ on $\mathcal{I}\setminus\mathcal{P}$ and only the sign of
$\mathcal{G}_rg$ essentially counts.

\subsection{Characterization of the Representing Function $\hat{f}$}
Let $y\in g^{-1}(\mathbb{R}_+)$ be given. Utilizing the known distribution function of $M$ yields (cf. \cite{BorSal15}, p. 26)
\begin{align*}
J_y(x)=\E_x\left[\hat{f}(M_T)\mathbbm{1}_{[y,b)}(M_T)\right]=\psi(x)\int_{x\lor y}^b\hat{f}(z)\frac{\psi'(z)}{\psi^2(z)}dz.
\end{align*}
Given this expression, it is now sufficient to find a function $\hat{f}$ for which the identity $V_{y}(x)=J_y(x)$ holds for all $x\in \mathcal{I}$. This identity holds for $x\geq y$ provided that
the {\em Volterra integral equation of the the first kind}
\begin{align}
\frac{g(x)}{\psi(x)}=\int_x^b\hat{f}(z)\frac{\psi'(z)}{\psi^2(z)}dz\label{kasvavaid}
\end{align}
is satisfied. Standard differentiation of identity \eqref{kasvavaid} now shows that
for all $x\in [y,b)\setminus \mathcal{P}$ we have
\begin{align}\label{eq f1}
\hat{f}(x)=g(x)-\psi(x)\frac{g'(x)}{\psi'(x)},
\end{align}
coinciding with the function $\rho$ derived in \cite{BaBa} by relying on functional concavity arguments.
Utilizing \eqref{canonical} demonstrates that this representing function can be alternatively be expressed as
\begin{align}\label{eq f1alt}
\hat{f}(x)= \frac{(L_\psi g)(x)}{(L_\psi \mathbbm{1})(x)}.
\end{align}
Consequently, if $\mathcal{G}_rg\in C(\mathcal{I}\backslash \mathcal{P})\cap\mathcal{L}_r^1(\mathcal{I})$, and
$\lim_{x\downarrow a}(L_\psi g)(x)=0$, then according to Lemma \ref{jointprobabilitydist} we have
\begin{align*}
\hat{f}(z)= -\frac{1}{r}\E_x[(\mathcal{G}_rg)(X_T)|M_T=z]
\end{align*}
for all $z\in \mathcal{I}$. Our first representation result is now stated
in the following theorem.
\begin{theorem}\label{theorem inc}
Fix $y\in g^{-1}(\mathbb{R}_+)$ and let $\hat{f}$ be as in \eqref{eq f1}.
Then, if $\lim_{x\rightarrow b-}g(x)/\psi(x)=0$, we have $J_y(x)=V_y(x)$. Moreover,
if $\hat{f}(x)$ is also nonnegative and nondecreasing and $g'(x)$ is lower semicontinuous for all $x\in[y,b)$, then $V_y(x)$ is $r$-excessive for $X$.
\end{theorem}
\begin{proof}
The first claim follows directly from identity \eqref{kasvavaid} after noticing that the representing function can be re-expressed for all $x\in \mathcal{I}\backslash \mathcal{P}$  as
\begin{align*}
\hat{f}(x)=-\frac{\psi^2(x)}{\psi'(x)}\frac{d}{dx}\left(\frac{g(x)}{\psi(x)}\right)
\end{align*}
and invoking the condition  $\lim_{x\rightarrow b-}g(x)/\psi(x)=0$.
Noticing that
since $g, \psi$, and $\psi'$ are continuous the lower semicontinuity of $g'$ on $[y,b)$ guarantees that $\hat{f}$ is upper semicontinuous on $[y,b)$ as well. If $\hat{f}$ is also nonnegative and nondecreasing, then $\hat{f}(x)\mathbbm{1}_{[y,b)}(x)$ is
nondecreasing, nonnegative, and upper semicontinuous  on $\mathcal{I}$. In that case $\hat{f}(M_T)\mathbbm{1}_{[y,b)}(M_T)=\sup_{t\in[0,T]}\{\hat{f}(X_t)\mathbbm{1}_{[y,b)}(X_t)\}$ and Proposition 2.1 in \cite{FoKn1} then guarantees that
$J_y(x)$ is $r$-excessive for $X$. Since $J_y(x)=V_y(x)$ the alleged result follows.
\end{proof}
Theorem \ref{theorem inc} shows that when $\hat{f}$ is chosen according to the rule \eqref{eq f1}
representation $J_y=V_y$ is valid provided that the limiting condition $\lim_{x\rightarrow b-}g(x)/\psi(x)=0$ is met.
Moreover, Theorem \ref{theorem inc} also shows that if $\hat{f}(x)\mathbbm{1}_{[y,b)}(x)$ is also nonnegative, nondecreasing, and upper semicontinuous, then the
representation is $r$-excessive for the underlying diffusion $X$. This observation is of interest since it demonstrates that the needed monotonicity of the representing function does not, in principle, require twice differentiability of the exercise payoff $g$. This is especially beneficial in the verification of the $r$-excessivity of a value since it essentially reduces the analysis into the analysis of the sign, monotonicity and semicontinuity of $\hat{f}$.
Note, however, that the representation needs not to majorize the exercise payoff and, therefore,
it does not necessarily coincide with the value of the considered stopping problem. Moreover, the required conditions for  $\hat{f}(x)\mathbbm{1}_{[y,b)}(x)$
are sufficient but {\em not necessary}  for the $r$-excessivity of $J_y$. As we will later see, there are circumstances where $J_y$ is $r$-excessive even
when $\hat{f}(x)\mathbbm{1}_{[y,b)}(x)$ is not monotonic.

An interesting comparative static result characterizing the sign of the relationship between increased volatility and the representing function $\hat{f}$ defined by \eqref{eq f1} is now summarized in the the next theorem.
\begin{theorem}\label{compstatthm}
Assume that the exercise reward $g$ is nondecreasing and that $\psi$ is convex. Then, increased volatility decreases or leaves
unchanged the value of the representing function $\hat{f}$ defined by \eqref{eq f1}. Moreover, increased volatility increases the expected value $\E_x[f(M_T)]$ for nondecreasing functions $f:\mathcal{I}\mapsto \mathbb{R}_+$ satisfying $f\in \mathcal{L}_r^1(\mathcal{I})$.
\end{theorem}
\begin{proof}
As shown in Lemma \ref{increasedvolatility}, the assumed convexity of $\psi$ guarantees that increased volatility decreases the logarithmic growth rate $\psi'/\psi$. Consequently, if the reward $g$ is nondecreasing on $\mathcal{I}$, then increased volatility increases the product $g'\psi/\psi'$ for all $x\in \mathcal{I}\backslash \mathcal{P}$. However, the assumed monotonicity of $g$ and the existence of the left- and right-hand limits at all $x\in \mathcal{P}$ demonstrates that increased volatility increases the product $g'(x\pm)\psi(x)/\psi'(x)$ for all $x\in \mathcal{P}$ as well. Applying this finding to the definition \eqref{eq f1} of the representing function $\hat{f}$ proves the first claim. On the other hand, utilizing the assumed monotonicity of the function $f$ in connection with Fubini's theorem yields
\begin{align*}
\E_x[f(M_T)] = f(x) + \int_x^b\p_x[M_T\geq v]df(v) = f(x) + \int_x^b\frac{\psi(x)}{\psi(v)}df(v)
\end{align*}
from which the alleged comparative static result follows.
\end{proof}
Theorem \ref{compstatthm} characterizes circumstances under which increased volatility unambiguously decreases or leaves unchanged the representing function $\hat{f}$ and increases or leaves unchanged the expected value of nondecreasing functions depending on the running supremum $M$. As we will later observe, both of these results have economically interesting consequences.

Having characterized the basic properties of the representing function $\hat{f}$, we are now in position to establish the following theorem connecting the representing function approach to standard sufficiency conditions.
\begin{theorem}\label{meidan esitys1}
Assume that the conditions of Lemma \ref{Martin1} are satisfied, that $\lim_{x\rightarrow b}g(x)/\psi(x)=0$, and that $g'(x)$ is lower semicontinuous on $[y^\ast,b)$. Then,
$$
V(x)=V_{y^\ast}(x)=J_{y^\ast}(x)=\E_x\left[\hat{f}(M_T)\mathbbm{1}_{[y^\ast,b)}(M_T)\right].
$$
\end{theorem}
\begin{proof}
It is clear that the conditions of the first claim of Theorem \ref{theorem inc} are satisfied. Consequently, $J_{y^\ast}(x)=V_{y^\ast}(x)$. The alleged result now follows from Lemma \ref{Martin1}.
\end{proof}
Theorem \ref{meidan esitys1} states a set of conditions under which the value of the optimal stopping strategy can be expressed as the expected value of the mapping $\hat{f}$ at the running maximum of the underlying diffusion. However, this does not yet guarantee that the value of the stopping problem could be expressed as an expected supremum since that requires in addition to the conditions of Theorem \ref{meidan esitys1} the monotonicity of $\hat{f}$. Moreover, Theorem \ref{meidan esitys1} relies on a set of sufficiency conditions based on the sign of $\mathcal{G}_rg$ and as such utilizes second order properties of the exercise payoff. Hence, it is of interest to investigate if at least part of the assumptions could be relaxed in the verification of optimality and the validity of the representation as an expected supremum. A set of sufficient conditions resulting into the desired outcome are summarized in our next theorem.
\begin{theorem}\label{meidan yleinen esitys1}
Assume that the exercise payoff $g$ is nondecreasing, that there is a unique interior threshold $y^\ast\in \mathcal{I}$ so that $\hat{f}(x\pm) \leq 0$ for $x\in (a,y^\ast)$ and $\hat{f}(x\pm) > 0$ for $x\in (y^\ast,b)$, that $g'$ is lower semicontinuous on $[y^\ast,b)$,
that $\hat{f}$ is nondecreasing on $[y^\ast,b)$, and that $g(x)/\psi(x)\downarrow 0$ as $x\uparrow b$. Then,
\begin{align}\label{expsup}
V(x)=V_{y^\ast}(x)=J_{y^\ast}(x)=\E_x\left[\sup_{t\in[0,T]}\{\hat{f}(X_t)\mathbbm{1}_{[y^\ast,b)}(X_t)\}\right]
\end{align}
for all $x\in \mathcal{I}$.
\end{theorem}
\begin{proof}
Since
$$
\frac{d}{dx}\left(\frac{g(x)}{\psi(x)}\right)=-\frac{\psi'(x)}{\psi^2(x)}\hat{f}(x)
$$
for all $x\in \mathcal{I}\backslash\mathcal{P}$, we notice that our assumptions on the sign of $\hat{f}$ guarantee that $g(x)/\psi(x)$ is increasing on $(a,y^\ast)$ and decreasing on $(y^\ast,b)$. Consequently, $y^\ast = \argmax\{g(x)/\psi(x)\}$ and $y^\ast\in\{x\in \mathcal{I}: V(x)=g(x)\}$ by Theorem 2.1 in \cite{ChIr11}. The monotonicity of $g$ and positivity of $\hat{f}$ on $(y^\ast,b)$ then imply that $y^\ast\in g^{-1}(\mathbb{R}_+)$. Since $g(x)/\psi(x)\downarrow 0$ as $x\uparrow b$, we find by utilizing Theorem \ref{theorem inc} that $V_{y^\ast}(x)=J_{y^\ast}(x)$ for all $x\in \mathcal{I}$. Moreover, the lower semicontinuity of $g'$ and monotonicity and positivity of $\hat{f}$ on $(y^\ast,b)$ guarantee that the conditions of the second claim of Theorem \ref{theorem inc} are satisfied and, therefore, that $V_{y^\ast}=J_{y^\ast}$ is $r$-excessive for $X$. Since $V_{y^\ast}$ majorizes the payoff $g$ for all $x\in \mathcal{I}$ and $V_{y^\ast}$ can be attained by utilizing the stopping strategy $\tau_{y^\ast}$ we notice that $V=V_{y^\ast}=J_{y^\ast}$. Finally the monotonicity of $\hat{f}$ implies that $\hat{f}(M_T)\mathbbm{1}_{[y^\ast,b)}(M_T)=\sup_{t\in[0,T]}\{\hat{f}(X_t)\mathbbm{1}_{[y^\ast,b)}(X_t)\}$ from which the alleged identity follows.
\end{proof}
Theorem \ref{meidan yleinen esitys1} states a set of conditions under which the value of the considered stopping problem admits a representation as an expected supremum. Instead of having to rely on the behavior of $\mathcal{G}_rg$, Theorem \ref{meidan yleinen esitys1} demonstrates that the verification of the optimality of a single boundary stopping strategy can be reduced to the study of the sing, monotonicity and sufficient regularity of the representing function $\hat{f}$. This observation is very useful especially in situations where the fundamental solution $\psi$ has a simple functional form since under such circumstances the verification of the validity of the conditions of Theorem \ref{meidan yleinen esitys1} is straightforward. However, as soon as $\psi$ takes more complicated forms, establishing the monotonicity of $\hat{f}$ becomes significantly more challenging and requires further analysis. In order to characterize relatively general circumstances
under which the function $\hat{f}$ is indeed monotonic, we first state the following auxiliary lemma.
\begin{lemma}\label{lemma mon1}
Let $y\in g^{-1}(\mathbb{R}_+)$ be given. Assume that either
\begin{itemize}
 \item[(A)] $g(x)$ is concave and $\psi(x)$ is convex on $[y,b)$, or
 \item[(B)]  there is a $z\in (a,y)$ so that $g(x)/\psi(x)$ is locally increasing at $z$,  $g'(x+)\leq g'(x-)$ for all $x\in (z,b)\cap\mathcal{P}$, and
 $(\mathcal{G}_rg)(x)$ is non-increasing and non-positive for all $x\in (z,b)$.
 \end{itemize}
Then, the function  $\hat{f}(x)$ characterized by \eqref{eq f1} is non-decreasing on $[y,b)$.
\end{lemma}
\begin{proof}
It is clear from \eqref{eq f1} that the required monotonicity of $\hat{f}$ is met provided that inequality
\begin{align}
\frac{d}{dx}\left(\frac{g'(x)}{\psi'(x)}\right)&<0\label{eq nec f1}
\end{align}
is satisfied for all $x\in [y,b)\setminus \mathcal{P}$ and
\begin{align}
\hat{f}(x+) -\hat{f}(x-) = \frac{g'(x-)-g'(x+)}{\psi'(x)}> 0\label{eq nec f2}
\end{align}
for all $x\in [y,b)\cap \mathcal{P}$. First, if $g$ is concave and $\psi$ is convex on $[y,b)$, then the inequalities \eqref{eq nec f1} and
\eqref{eq nec f2} are satisfied and $g'(x)/\psi'(x)$ is non-increasing on $[y,b)$ as claimed.
Assume now instead that the conditions of part (B) are satisfied. It is clear that since $[y,b)\subset(z,b)$ \eqref{eq nec f2} is satisfied by assumption
for all $x\in [y,b)\cap\mathcal{P}$. On the other hand, standard differentiation shows that for all $x\in(z,b)\setminus \mathcal{P}$
\begin{align*}
\frac{d}{dx}\left(\frac{g'(x)}{\psi'(x)}\right)=\frac{S'(x)}{{\psi'}^2(x)}\left[\frac{g''(x)}{S'(x)}\psi'(x)-\frac{\psi''(x)}{S'(x)}g'(x)\right]
=\frac{2S'(x)\mathcal{D}(x)}{\sigma^2(x){\psi'}^2(x)}.
\end{align*}
where
$$
\mathcal{D}(x)=(\mathcal{G}_rg)(x)\frac{\psi'(x)}{S'(x)}+r(L_\psi g)(x).
$$
The assumed monotonicity and non-positivity of $(\mathcal{G}_rg)(x)$ on $(z,b)\setminus \mathcal{P}$ and identity \eqref{canonical} now implies that
\begin{align*}
\mathcal{D}(x)&=(\mathcal{G}_rg)(x)\frac{\psi'(x)}{S'(x)}-r\int_z^x\psi(v)(\mathcal{G}_rg)(v)m'(v)dv+r(L_\psi g)(z+)\\
&\leq (\mathcal{G}_rg)(x)\frac{\psi'(z)}{S'(z)}+r(L_\psi g)(z+)\leq r(L_\psi g)(z+)
\end{align*}
for all $x\in(z,b)\setminus \mathcal{P}$. However, the assumed monotonicity of $g(x)/\psi(x)$ in a neighborhood of $z$ then guarantees that
$(L_\psi g)(z+) \leq 0$, proving that $\mathcal{D}(x)\leq 0$ for all $x\in(z,b)\setminus \mathcal{P}$.
\end{proof}
Lemma \ref{lemma mon1} states a set of conditions under which the function  $\hat{f}$ characterized by \eqref{eq f1} is non-decreasing on the set $[y,b)$ and, therefore, the function
$\hat{f}(x)\mathbbm{1}_{[y,b)}(x)$ is nondecreasing on $\mathcal{I}$. Interestingly, the first of these conditions is based solely on the concavity of the exercise payoff and the convexity of the increasing fundamental
solution without imposing further requirements. Since the
convexity of the fundamental solution $\psi$ is determined by $\mu$ and $\sigma$, part (A) of Lemma \ref{lemma mon1} essentially delineates circumstances under which the monotonicity
of the representing function $\hat{f}$ could be, in principle,
characterized solely based on the infinitesimal characteristics of the underlying diffusion and the concavity of the exercise payoff.
Part (B) of Lemma \ref{lemma mon1} shows, in turn, how the monotonicity of the  function  $\hat{f}$ is associated with the monotonicity of the generator $\mathcal{G}_rg$.
The conditions of part (B) of Lemma \ref{lemma mon1} are satisfied, for example,  under the assumptions of Remark \ref{sufficient} provided that
$\mathcal{G}_rg$ is non-increasing on $(\tilde{x},b)$ and $z\in (\tilde{x}, y\wedge y^\ast)$.

Moreover, it is clear that under the conditions of Lemma \ref{lemma mon1} we have $J_y(x)=V_y(x)$ for all $x,y\in\mathcal{I}$.
However, without imposing further restrictions on the behavior of the payoff we do not know whether $\hat{f}(x)\mathbbm{1}_{[y,b)}(x)$ generates
the smallest $r$-excessive majorant of the exercise payoff $g$ or not, nor do we know how $\hat{f}(x)\mathbbm{1}_{[y,b)}(x)$ behaves in the neighborhood of the optimal
stopping boundary. Our next theorem summarizes a set of conditions under which these questions can be unambiguously answered.
\begin{theorem}\label{thm optimal condition}
Assume that there is a unique interior threshold $y^\ast=\inf\{x\in \mathcal{I}:\hat{f}(x)>0\}\in \mathcal{I}$, that the conditions (A) or (B) of Lemma \ref{lemma mon1} are satisfied on $[y^\ast,b)$, and that
$g'$ is lower semicontinuous on $[y^\ast,b)$.
Then,  $\hat{f}(y^\ast)=0$ if $y^\ast\in \mathcal{I}\setminus \mathcal{P}$ and
$$
\hat{f}(y^\ast) = g(y^\ast)-\frac{\psi(y^\ast)}{\psi'(y^\ast)}g'(y^\ast+)> 0
$$ if $y^\ast\in \mathcal{P}$. Moreover,
\begin{align}
\hat{f}(x) = \frac{(L_\psi g)(x+)}{(L_\psi \mathbbm{1})(x)} = \frac{(L_\psi g)(y^\ast+)-\int_{y^\ast}^{x}(\mathcal{G}_rg)(v)\psi(v)m'(v)dv}{(L_\psi \mathbbm{1})(x)}\label{esitys1}
\end{align}
for all $x\in (y^\ast,b)\setminus \mathcal{P}$, and
\begin{align}\label{optstopval}
\begin{split}
V(x)&=V_{y^{\ast}}(x)=J_{y^\ast}(x)=\psi(x)\sup_{y\geq x}\left[\frac{g(y)}{\psi(y)}\right]=\psi(x)\frac{g(x\lor y^\ast)}{\psi(x\lor y^\ast)}\\
&=
\E_x\left[\sup_{t\in[0,T]}\hat{f}(X_t)\mathbbm{1}_{[y^\ast,b)}(X_t)\right].
\end{split}
\end{align}
\end{theorem}
\begin{proof}
Claim \eqref{esitys1} follows from the identity $\hat{f}(x)=\frac{S'(x)}{\psi'(x)}(L_\psi g)(x)$ by invoking the canonical form \eqref{canonical}.
The rest of the claims follow directly from Theorem \ref{meidan yleinen esitys1}.
\end{proof}
Theorem \ref{thm optimal condition} shows that the continuity of the function $\hat{f}$ at the optimal boundary $y^\ast$ coincides with the standard
{\em smooth fit principle} requiring that the value should be continuously differentiable across the optimal boundary. However, as is clear from
Theorem \ref{thm optimal condition}, if the optimal boundary is attained at a threshold where the exercise payoff is not differentiable, then $\hat{f}$
is discontinuous at the optimal boundary $y^\ast$. Furthermore, since the nonnegativity and monotonicity of $\hat{f}(x)\mathbbm{1}_{[y^\ast,b)}(x)$ on $[y^\ast,b)$ are sufficient
for the validity of Theorem \ref{thm optimal condition}, we observe in accordance with the results by \cite{ChSaTa} that $\hat{f}(x)\mathbbm{1}_{[y^\ast,b)}(x)$ is
only upper semicontinuous on $\mathcal{I}$.

Theorem \ref{thm optimal condition} also shows that $\hat{f}(x)$ has a neat integral representation \eqref{esitys1} capturing the size of the potential
discontinuity of $\hat{f}(x)$ at $y^\ast$. In the case where $a$ is unattainable and the smooth fit principle is satisfied at $y^\ast$ \eqref{esitys1}
can be re-expressed as  (cf. Proposition 2.13 in \cite{ChSaTa})
\begin{align}
\hat{f}(x) = -\frac{\int_{y^\ast}^{x}(\mathcal{G}_rg)(v)\psi(v)m'(v)dv}{r\int_{a}^{x}\psi(v)m'(v)dv}\label{esitys1b}
\end{align}
and, hence, in that case the value reads as
\begin{align}
V(x) = -\mathbb{E}_x\left[\frac{\int_{y^\ast}^{M_T}(\mathcal{G}_rg)(v)\psi(v)m'(v)dv}{r\int_{a}^{M_T}\psi(v)m'(v)dv}\mathbbm{1}_{[y^\ast,b)}(M_T)\right]\label{esitys1c}
\end{align}
It is clear that if the sufficient conditions stated in Remark \ref{sufficient} are satisfied, and in addition $(\mathcal{G}_rg)(x)$ is non-increasing on $(y^\ast,b)$,
and $a$ is unattainable for the underlying diffusion, then the conditions of  Theorem \ref{thm optimal condition} are met and
\begin{align*}
V(x) = \mathbb{E}_x\left[\sup_{t\in [0,T]}\left(-\frac{\int_{y^\ast}^{X_t}(\mathcal{G}_rg)(v)\psi(v)m'(v)dv}{r\int_{a}^{X_t}\psi(v)m'(v)dv}\mathbbm{1}_{[y^\ast,b)}(X_t)\right)\right].
\end{align*}

\begin{remark}
Our approach relies on the identity $$\p_x[M_T\geq y]=\p_x[\tau_y<T]=\E_{x}\left[e^{-r\tau_y}\right] = \frac{\psi(x)}{\psi(y)}, \quad a<x < y <b$$
which is essentially based on the continuity of the running supremum process $M_t$. Since the running supremum of a spectrally negative jump-diffusion is continuous as well and a jump-diffusion is a Hunt process, we notice that our principal findings on the representing function are valid for that class of processes as well provided that a set of sufficient regularity conditions are met (cf. \cite{AlMaRa14}). In that setting the increasing fundamental solution $\psi$ can be identified as the $r$-scale function associated with the particular spectrally negative jump diffusion (cf. Theorem 8.1 in \cite{Kyprianou06}).
\end{remark}

\section{Illustrations and Extensions}
We now illustrate our general findings in five separate examples in order to illustrate the applicability of the developed approach as well as the intricacies associated with the considered representation. The first example focuses on a stopping problem arising in the literature on economic mechanism design.
In the second example we reconsider the analysis of an optimal stopping signal originally studied in \cite{BaBa} and connect it to the analysis developed in our manuscript.
The third example focuses, in turn, on a case where the payoff is smooth and the stopping strategy is of
the single boundary type. Despite these favorable properties, we will show that it does not always result into a value characterizable as an expected supremum in the spirit of \eqref{prob2}.
The fourth example, in turn, focuses on a less smooth case resulting into a representation where the representing function is monotone but not everywhere continuous. Finally, the fifth example focuses on spectrally negative jump diffusions and show how the developed approach applies there as well.

\subsection{Incentive Compatible Implementable Stopping Rules}

\cite{KrST15} and \cite{KrSt15b} consider the determination of incentive compatible implementable stopping rules arising in economic studies analyzing mechanism design. One of the key questions within the framework developed in \cite{KrST15} and \cite{KrSt15b} is to investigate if there exits a {\em transfer} which would result into the optimality of a desired exercise strategy characterized by a so-called cut-off rule. Such problems arise quite naturally, for example, in models considering situations where individual exercise strategies do not coincide with a socially desirable exercise rule.  In such cases the decision making problem of a social planner can be reduced into the determination of a transfer rule (for example, a tax) resulting into the individual optimality of the socially desirable state. As we will now demonstrate, the approach developed in this study is particularly appropriate for the analysis of this question within the considered infinite horizon setting. To see that this is indeed the case, assume for simplicity that the exercise payoff $g:\mathcal{I}\mapsto \mathbb{R}$ is continuously differentiable on $\mathcal{I}$. Assume also that the representing function $\hat{f}$ defined for $x\in \mathcal{I}$ by
$\hat{f}(x)=g(x)-g'(x)\psi(x)/\psi'(x)$ is nondecreasing and changes uniquely sign at the interior threshold $y^\ast = \hat{f}^{-1}(0) \in (a, b)$. It is clear from Theorem \ref{thm optimal condition} that in that case $y^\ast=\argmax\{g(x)/\psi(x)\}$, $\tau_{y^\ast}=\inf\{t\geq 0:X_t\geq y^\ast\}$ is an optimal stopping time, and the value reads as in \eqref{optstopval}. Given these observations, we now consider the associated optimal stopping problem
\begin{align}\label{socialplan}
V^{\hat{f}}(x)=\sup_\tau \E_x\left[e^{-r\tau}(g(X_\tau)-\hat{f}(k^\ast))\right].
\end{align}
where $k^\ast\in \mathcal{I}$ is an exogenously set threshold. We now find that our assumptions guarantee the following result.
\begin{proposition}\label{incentives}
Under our assumptions,
$$k^\ast=\argmax\left\{\frac{g(x)-f(k^\ast)}{\psi(x)}\right\},$$ $\tau_{k^\ast}=\inf\{t\geq 0:X_t\geq k^\ast\}$ is an optimal stopping time, and the value reads as
\begin{align}\label{valuesocialplan}
\begin{split}
V^{\hat{f}}(x) &= \E_x\left[e^{-r\tau_{k^\ast}}(g(X_{\tau_{k^\ast}})-\hat{f}(k^\ast))\right]=\psi(x)\;\frac{g(x\lor k^\ast)-\hat{f}(k^\ast)}{\psi(x\lor k^\ast)}\\
&=\E_x\left[\sup_{t\in[0,T]}\left(\hat{f}(X_t\lor k^\ast)-\hat{f}(k^\ast)\right)\right].
\end{split}
\end{align}
\end{proposition}
\begin{proof}
Consider now the function $f(x):=\hat{f}(x)-\hat{f}(k^\ast)$. It is clear that our assumptions guarantee that $f$ vanishes at $k^\ast$ and is continuous and nondecreasing on $\mathcal{I}$.
On the other hand, the representing function of the exercise payoff $\hat{g}(x):=g(x)-\hat{f}(k^\ast)$ reads as
$$
\hat{g}(x) - \hat{g}'(x)\frac{\psi(x)}{\psi'(x)}= \hat{f}(x)-\hat{f}(k^\ast) = f(x).
$$
The alleged result now follows from Theorem  \ref{thm optimal condition}.
\end{proof}

Proposition \ref{incentives} states a set of conditions under which the representing function $\hat{f}$ can be utilized for characterizing a transfer rule resulting into the optimality of a desired fixed exercise threshold. This result is interesting since it essentially delineates circumstances under which a social planner can shift the individually optimal exercise threshold of a decision maker to a different socially optimal level by simply subtracting (or adding) the value of the representing function at the desired state to the exercise payoff. As we will notice in the following section, this result is closely related
with the literature on optimal stopping signals and Gittins indices as well. A nice comparative static implication of our Lemma \ref{increasedvolatility} is summarized in the following.
\begin{corollary}\label{corinc}
Assume that there is a $\tilde{y}\in [a,b]$ so that $\mu(x)-rx$ is non-increasing on $(a,\tilde{y})$, $\mu(x)\leq 0$ for $x\in(\tilde{y},b)$ and $\lim_{x\downarrow a}\mu(x)\leq 0$ whenever $a$ is attainable for the underlying diffusion. Then increased volatility decreases or leaves unchanged the transfer $\hat{f}(k^\ast)$.
\end{corollary}
\begin{proof}
The alleged result is a direct consequence of Lemma \ref{increasedvolatility} after noticing that the assumptions guarantee that $\psi$ is strictly convex on $\mathcal{I}$ (see Lemma 3.3 in \cite{Alvarez04}).
\end{proof}

\subsection{Optimal Stopping Signals}

We now proceed in our illustrations and show how the developed representation approach is related with non-standard stopping problems arising in the analysis of Gittins indices and optimal stopping signals (cf., for example, \cite{BaBa}, \cite{BaFo}, \cite{ElKaFo}, \cite{ElKarKar94}, \cite{KasMan1998}, and \cite{Mandelbaum87}). In line with the notation in \cite{BaBa}, we assume that $k\in \mathbb{R}$ is an exogenously given parameter and let
\begin{align}
V_k(x):= \sup_{\tau}\E_x\left[e^{-r\tau}(g(X_\tau)-k)\right]\label{babaprob}
\end{align}
denote the value of the considered optimal stopping problem and assume that the exercise payoff is continuously differentiable on $\mathcal{I}$. As in \cite{BaBa} we also assume that the boundaries are natural for the underlying diffusion $X$. This guarantees that even though the process may tend towards a boundary, it will never attain it in finite time.  In connection with problem \eqref{babaprob}, we also consider the associated non-standard stopping problem
\begin{align}\label{gittins}
\gamma(x) = \inf_{\tau \in \mathcal{T}}\frac{\E_x[g(x)-e^{-r\tau}g(X_\tau)]}{1-\E_x[e^{-r\tau}]},
\end{align}
where $\mathcal{T}$ denotes the class of firs exit times from open subsets of $\mathcal{I}$ with compact closure in $\mathcal{I}$. As was established in \cite{BaBa}, the stopping region $\Gamma_k=\{x\in \mathcal{I}:V_k(x)=g(x)-k\}$ for the problem \eqref{babaprob} coincides with the set $\{x\in \mathcal{I}:\gamma(x)\geq k\}$. We now plan to show how these results can be replicated by utilizing the representation result developed in our paper.

It is clear from our analysis that in the present case it is sufficient that for $k\in \mathbb{R}$ the representing function
\begin{align}
\hat{f}_k(x)= g(x) - k - \psi(x)\frac{g'(x)}{\psi'(x)} = \hat{f}_0(x) - k\label{repgitt}
\end{align}
is nondecreasing and changes uniquely sign at the interior threshold $y^\ast_k = \hat{f}^{-1}_k(0) \in (a, b)$ satisfying equation $\hat{f}_0(y^\ast_k) = k$.
If that is the case, then
\begin{align*}
V_k(x)= \E_x\left[\sup_{t\in[0, T]}\hat{f}_k(X_t\lor y^\ast_k)\right] = \E_x\left[\sup_{t\in[0, T]}\left(\hat{f}_0(X_t\lor y^\ast_k)-k\right)\right]
\end{align*}
and $\Gamma_k=\{x\in \mathcal{I}: \hat{f}_0(x)\geq k\}$. As was established in Theorem 13 of \cite{BaBa}, $\gamma(x)=\hat{f}_0(x)$ in the present single boundary setting (see also Section 3.10 in \cite{ElKarKar94} for the decreasing case). Consequently, we notice that the considered supremum representation results in the correct expressions for the considered functionals. Moreover, given the identity $\gamma(x)=\hat{f}_0(x)$ we observe the following interesting comparative static property of the the value \eqref{gittins}:
\begin{corollary}
Assume that the condition of Corollary \ref{corinc} are satisfied. Then increased volatility increases the value \eqref{gittins}.
\end{corollary}
\begin{proof}
Analogous with the proof of Corollary \ref{corinc}.
\end{proof}

We would like to point out at that the determination of incentive compatible implementable stopping rules considered in the previous subsection is closely associated with the present case as well. To see that this is indeed the case, we immediately notice that if the representing function $\hat{f}_0$ is monotonically increasing then choosing $k=\hat{f}_0(z^\ast)$ for some fixed threshold $z^\ast\in \mathcal{I}$ implies that $\Gamma_{\hat{f}_0(z^\ast)}=\{x\in \mathcal{I}:\gamma(x)\geq \hat{f}_0(z^\ast)\}=\{x\in \mathcal{I}:\hat{f}_0(x)\geq \hat{f}_0(z^\ast)\}=[z^\ast,b)$.

Finally, it is also worth noticing that the non-standard stopping problem \eqref{gittins} has in many cases an interesting interpretation as an appropriate maximal conditional expectation. To see that this is indeed the case, denote by $\Theta$ the set of functions $g:\mathcal{I}\mapsto \mathbb{R}$ belonging into the domain of the extended operator of the underlying process $X$ killed at $T$ and satisfying  for $\tau\in \mathcal{T}$ the generalized Dynkin formula (see, for example, \cite{ChSaTa}, \cite{CrMo14},\cite{HeSt10}, and \cite{LaZe13})
\begin{align}
\E_x\left[e^{-r\tau}g(X_\tau)\right]=g(x)+\E_x\left[\int_0^\tau e^{-rs}\tilde{g}(X_s)ds\right],\label{Dynkin}
\end{align}
where $\tilde{g}\in \mathcal{L}_r^1(\mathcal{I})$ naturally coincides with the generator $(\mathcal{G}_rg)(x)$ whenever the payoff is sufficiently smooth. It is now clear that in this case \eqref{gittins} can be re-expressed as
\begin{align}\label{gittinsgen}
\gamma(x) = -\sup_{\tau \in \mathcal{T}}\frac{\E_x\int_0^\tau e^{-rs}\tilde{g}(X_s)ds}{r\E_x\int_0^\tau e^{-rs}ds}=-\sup_{\tau \in \mathcal{T}}\frac{1}{r}\E_x\left[\tilde{g}(X_T)|T<\tau\right].
\end{align}
Especially, if the exercise payoff constitutes an expected cumulative present value of a flow and reads as $g(x)=(R_r\pi)(x)$ for some continuous $\pi\in \mathcal{L}_r^1(\mathcal{I})$ then $g\in \Theta$ and the non-standard stopping problem \eqref{gittins} can be re-expressed in a more familiar form as
\begin{align*}
\gamma(x) = \sup_{\tau \in \mathcal{T}}\frac{\E_x\int_0^\tau e^{-rs}\pi(X_s)ds}{r\E_x\int_0^\tau e^{-rs}ds}
\end{align*}
implying along the lines of our Lemma \ref{jointprobabilitydist} that
\begin{align}\label{gittinsmod}
\gamma(x)= \sup_{\tau \in \mathcal{T}}\frac{1}{r}\E_x\left[\pi(X_T)|T<\tau\right] = \hat{f}_0(x).
\end{align}
Consequently, we notice that in this case the representing function can be interpreted as the maximal expected present value of the cash flow at the independent exponential terminal date provided that the process is still alive at that instant.

\subsection{Optimal Entry}
In order to illustrate circumstances where the value of a single boundary problems cannot necessarily be expressed as an expected supremum, we now assume that the upper boundary $b$ is unattainable for $X$ and that the exercise payoff can be expressed as an expected cumulative present value
$g(x) = (R_r\pi)(x)$ for some continuous revenue flow $\pi\in \mathcal{L}_r^1(\mathcal{I})$ satisfying the conditions $\pi(x)\gtreqqless 0$ for $x \gtreqqless x_0$, where $x_0\in (a,b)$,
$\lim_{x\downarrow a}\pi(x)< -\varepsilon$ and
$\lim_{x\uparrow b}\pi(x) > \varepsilon$ for some $\varepsilon > 0$. This type of models arise frequently in studies considering optimal entry under uncertainty.

It is clear that under these conditions the exercise payoff satisfies the conditions $g\in C^2(\mathcal{I})$ and
$(\mathcal{G}_rg)(x)= -\pi(x)\lesseqqgtr 0$ for $x \gtreqqless x_0$. Moreover, utilizing representation \eqref{Green} shows that
in the present case
\begin{align*}
(L_\psi g)(x)=\int_a^x\psi(t)\pi(t)m'(t)dt
\end{align*}
Our assumptions guarantee that $(L_\psi g)(x)<0$ for all $x\leq x_0$ and that $(L_\psi g)(x)$ is monotonically increasing on $(x_0,b)$. Fix $x_1>x_0$. Then a
standard application of the mean value theorem for definite integrals
yields
\begin{align*}
(L_\psi g)(x)&=(L_\psi g)(x_1)+\int_{x_1}^x\psi(t)\pi(t)m'(t)dt \\
&= (L_\psi g)(x_1)+\frac{\pi(\xi)}{r}\left[\frac{\psi'(x)}{S'(x)}-\frac{\psi'(x_1)}{S'(x_1)}\right],
\end{align*}
where $\xi\in(x_1,x)$. Letting $x\rightarrow b$ and noticing that $\psi'(x)/S'(x)\rightarrow \infty$ as $x\rightarrow b$ (since $b$ was assumed to be unattainable
for $X$, cf. p. 19 in \cite{BorSal15}) then shows that
$\lim_{x\uparrow b}(L_\psi g)(x)=\infty$ proving that equation $(L_\psi g)(x)=0$ has a unique root $y^\ast\in(x_0,b)$ and that $y^\ast = \argmax\{(R_r\pi)(x)/\psi(x)\}$. Moreover, the value
\eqref{eq prob} can be expressed as
$$
V(x) = \psi(x)\frac{(R_r\pi)(x\lor y^\ast)}{\psi(x\lor y^\ast)} = \begin{cases}
(R_r\pi)(x) &x\geq y^\ast\\
\frac{(R_r\pi)(y^\ast)}{\psi(y^\ast)}\psi(x) &x<y^\ast.
\end{cases}
$$

The representing function $\hat{f}(x)$ characterized in Theorem \ref{theorem inc} can be expressed in the present setting as
$$
\hat{f}(x)=\frac{S'(x)}{\psi'(x)}\int_{a}^x\psi(y)\pi(y)m'(y)dy.
$$
As was established in Theorem \ref{thm optimal condition}, we have that $\hat{f}(y^\ast)=0$ and
$$
V(x)=\E_x\left[\frac{S'(M_T)}{\psi'(M_T)}\int_{y^{\ast}}^{M_T\lor y^\ast}\psi(y)\pi(y)m'(y)dy\right].
$$
Moreover,
standard differentiation shows that for all $x\in(y^\ast,b)$ we have
$$
\hat{f}'(x) = \frac{2S'(x)\psi(x)}{{\psi'}^2(x)\sigma^2(x)}\left[\pi(x)\frac{\psi'(x)}{S'(x)}-r\int_{y^\ast}^{x}\psi(t)\pi(t)m'(t)dt\right]
$$
demonstrating that $\hat{f}$ is nondecreasing for $x\in(y^\ast,b)$ only if
$$
\pi(x)\frac{\psi'(x)}{S'(x)}\geq r\int_{y^\ast}^{x}\psi(t)\pi(t)m'(t)dt
$$
for all $x\geq y^\ast$. Otherwise it is clear from our results that the value of the considered optimal stopping problem {\em cannot} be expressed as an expected supremum
of the form \eqref{prob2} (see Figure \ref{fig example 1}(a)).
A simple sufficient condition guaranteeing the required monotonicity is to assume that $\pi(x)$ is nondecreasing on $(x_0,b)$ since in that case we have
\begin{align*}
\hat{f}'(x)&\geq\frac{2S'(x)\psi(x)}{{\psi'}^2(x)\sigma^2(x)}\left[\pi(x)\frac{\psi'(x)}{S'(x)}-r\pi(x)\int_{y^\ast}^{x}\psi(t)m'(t)dt\right]\\
&\geq\frac{2S'(x)\psi(x)}{{\psi'}^2(x)\sigma^2(x)}\pi(x)\frac{\psi'(y^\ast)}{S'(y^\ast)}\geq 0.
\end{align*}
If this is indeed the case, then
$$V(x) = \mathbb{E}_x\left[\sup_{t\in[0,T]}\frac{S'(X_t\lor y^\ast)}{\psi'(X_t\lor y^\ast)}\int_{y^\ast}^{X_t\lor y^\ast}\psi(v)\pi(v)m'(v)dv\right].$$

\subsection{Capped Call Option}
In order to illustrate our findings in a nondifferentiable setting, assume now that the upper boundary $b$ is unattainable for $X$ and that the exercise
payoff $g(x) = \min((x-K)^+, C)$, with $a < K < K+C < b$, satisfies the limiting inequality
\begin{align}
\lim_{x\downarrow a}\frac{|x-K|}{\varphi(x)}<\infty.\label{limit}
\end{align}
Assume also that the
appreciation rate $\theta(x)=\mu(x)-r(x-K)$ satisfies the conditions $\theta\in\mathcal{L}^1_r(\mathcal{I})$, $\theta(x)\gtreqqless 0$ for $x\lesseqqgtr x_0^\theta$,
where $x_0^\theta\in \mathcal{I}$,
and $\lim_{x\rightarrow b}\theta(x) < -\varepsilon$ for $\varepsilon >0$.

We notice that the exercise payoff $g$ is continuous, nondecreasing, and twice continuously differentiable on $\mathcal{I}\backslash\{K,K+C\}$. Moreover,
$g'(K-)\leq g'(K+)$, $\lim_{x\rightarrow (K+C) -}g'(x)\geq \lim_{x\rightarrow (K+C) +}g'(x)$, and
$$
(\mathcal{G}_rg)(x)=\begin{cases}
-rC, &x\in(K+C,b)\\
\theta(x), &x\in(K,K+C)\\
0,&\in(a,K).
\end{cases}
$$
It is now clear that the conditions of Remark \ref{sufficient} are satisfied. Thus, we known that there exists a unique optimal exercise threshold
$x^\ast =\argmax\{g(x)/\psi(x)\}$ and $V(x)=V_{x^\ast}(x)$. Our objective is now to prove that
this threshold reads as
$x^\ast = \min(C+K,y^\ast)$,
where $y^\ast > x_0^\theta$ is the unique root of the ordinary first order condition
$$
\psi(y^\ast)=\psi'(y^\ast)(y^\ast-K).
$$
To see that this is indeed the case, we first observe by applying part (A) of Corollary 3.2 in \cite{Al04} combined with the limiting condition \eqref{limit} that
$$
\frac{\psi^2(x)}{S'(x)}\frac{d}{dx}\left[\frac{x-K}{\psi(x)}\right] =\frac{\psi(x)}{S'(x)}-(x-K)\frac{\psi'(x)}{S'(x)}=\int_a^x\psi(t)\theta(t)m'(t)dt-\frac{a-K}{\varphi(a)}.
$$
Applying analogous arguments with the ones in Example 3, we find that equation
$$
\int_a^x\psi(t)\theta(t)m'(t)dt-\frac{a-K}{\varphi(a)}=0
$$
has a unique root $y^\ast\in (x_0^\theta,b)$ so that $y^\ast=\argmax\{(x-K)/\psi(x)\}$. Moreover,
$$
U(x)=\sup_{\tau}\mathbb{E}_x\left[e^{-r\tau}(X_\tau-K)^+\right]=\begin{cases}
x-K &x\geq y^\ast\\
(y^\ast-K)\frac{\psi(x)}{\psi(y^\ast)} &x<y^\ast.
\end{cases}
$$

In light of these observations, we find that if $y^\ast\in (K,K+C)$, then it is sufficient to notice that $V_{x^\ast}(x)=\min(C,U(x))$ is $r$-excessive since
constants are $r$-excessive and $U(x)$ is also $r$-excessive. Moreover, since
both $C$ and $U(x)$ dominate the payoff, we notice that $V_{x^\ast}(x)=\min(C,U(x))$ constitutes the smallest $r$-excessive majorant of $g(x)$ and, therefore, $V(x)=V_{x^\ast}(x)=\min(C,U(x))$.
If instead $y^\ast\geq K+C$, then $x^\ast=K+C=\argmax\{g(x)/\psi(x)\}$ and the optimal policy is to follow the stopping policy $\tau_{x^\ast}=\inf\{t\geq 0: X_t\geq K+C\}$ with a value
$$
\tilde{U}(x)=C\mathbb{E}_x\left[e^{-r\tau_{x^\ast}}\right]=\begin{cases}
C &x\geq C+K\\
C\frac{\psi(x)}{\psi(C+K)} &x<C+K.
\end{cases}
$$
Given these findings, we notice that if $y^\ast\geq K+C$, then $x^\ast=K+C$ and
$$
f(x) = C\mathbbm{1}_{[x^\ast,b)}(x)\geq 0
$$
is nonnegative and nondecreasing and, consequently,
$$
V(x)=C\mathbb{E}_x\left[\mathbbm{1}_{[x^\ast,b)}(M_T)\right]=C\mathbb{P}_x\left[M_T\geq K+C\right].
$$
However, since $f(x^\ast-)=0$ and $f(x^\ast+)=C$ we notice that $f$ is discontinuous at the optimal threshold $x^\ast$ (see Figure \ref{fig example 1}(b)).
If $y^\ast<K+C$, then the nonnegative function
$$
f(x)=\begin{cases}
C &x\geq C+K\\
x-K-\frac{\psi(x)}{\psi'(x)} &x\in[y^\ast,K+C)
\end{cases}
$$
in nondecreasing only if the increasing fundamental solution is convex on $(y^\ast,K+C)$ (it has to be locally convex at $y^\ast$). If the convexity requirement is met, then
$$
V(x)=\mathbb{E}_x\left[\left(M_T-K-\frac{\psi(M_T)}{\psi'(M_T)}\right)\mathbbm{1}_{[y^\ast,C+K)}(M_T)\right]+C\mathbb{P}_x\left[M_T\geq C+K\right].
$$
Moreover, since
$f(C+K+)=C > C-\frac{\psi(C+K-)}{\psi'(C+K-)} = f(C+K-)$, we notice that $f$ is discontinuous at $C+K$.

\begin{figure}[!ht]
\begin{center}
\begin{subfigure}[b]{0.4\textwidth}
\begin{center}
\includegraphics[width=\textwidth]{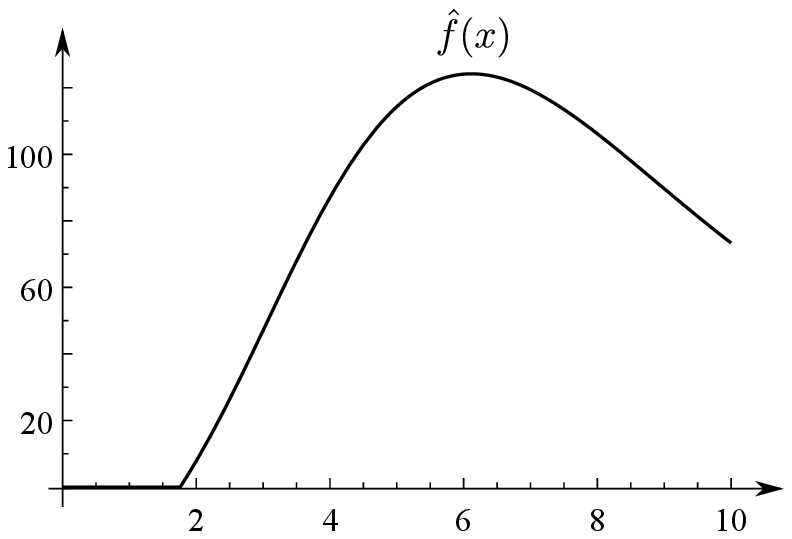}
\end{center}
\caption{\tiny Example 1: Smooth payoff with $\pi(x)=(x^5-2)e^{-x}+1$ leads to a non-increasing $\hat{f}$. In this case the representation as an expected supremum fails to exist.}\label{fig example 1a}
\end{subfigure}
\qquad
\begin{subfigure}[b]{0.4\textwidth}
\begin{center}
\includegraphics[width=\textwidth]{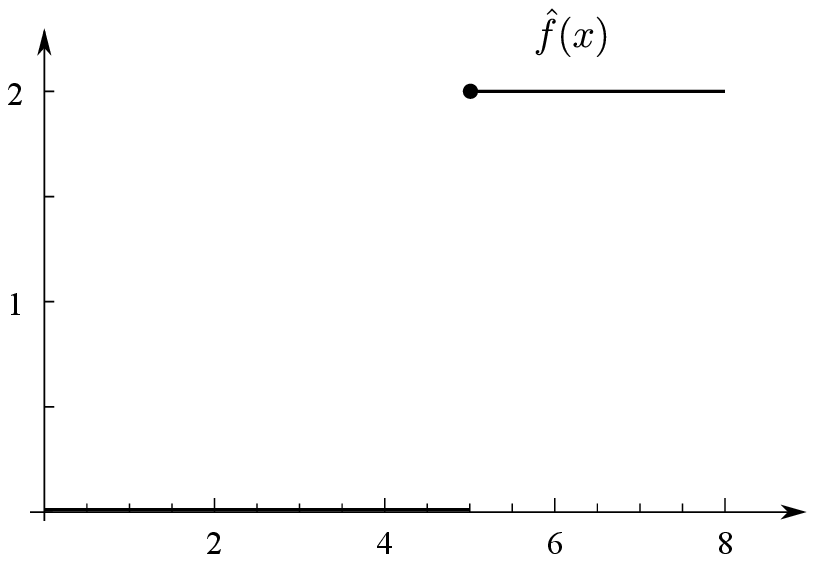}
\end{center}
\caption{\tiny Example 2: Capped call option with $g(x)=\min\{(x-3)^+,2\}$ leads to a discontinuous $\hat{f}$.\\ }\label{fig example 1b}
\end{subfigure}
\end{center}
\caption{\small Numerical examples based on geometric Brownian motion. Parameters have been chosen such that $\psi=x^2$ and $\varphi=x^{-4}$}\label{fig example 1}
\end{figure}

\subsection{Spectrally Negative Jump-diffusions}
In order to illustrate our findings for a spectrally negative jump diffusion, consider now the geometric Lévy process $X=\{X_{t}\}$ with a finite Lévy measure
$\nu=\lambda \mathfrak{m}$, where $\mathfrak{m}$ denotes the jump size distribution,  characterized by the dynamics
\begin{align*}
d X_{t}=X_{t-}\left\{\mu dt+\sigma d W_{t}+\lambda
\int_{(0,1)} z \tilde{N}(dt,dz) \right\},\quad X_0:=x\in \mathbb{R}_+
\end{align*}
where $\mu\in \mathbb{R}$ and $\sigma>0$. It is now a straightforward exercise to show that $\psi(x)=x^{\varrho}$, where
$\varrho>0$ denotes the positive root of the characteristic equation
\begin{align*}
&&\frac{1}{2} \sigma^{2} \varrho(\varrho-1)+(\mu+\lambda
\bar{m})\varrho-(r+\lambda)+\lambda \int_{0}^{1}(1-z)^\varrho
\mathfrak{m}(dz)=0,
\end{align*}
and
$$
\bar{m}=\int_0^1z\mathfrak{m}(dz)
$$
denotes the expected jump size. We observe that if $$\hat{f}(x)=g(x)-\frac{1}{\varrho}g'(x)x$$ is nondecreasing, there is an interior point $y^\ast=\inf\{x\geq0:\hat{f}(x)\geq 0\}\in (0,\infty)$, and $g'$ is lower semicontinuous on $[y^\ast,\infty)$, then
$$
V(x)=\E_x\left[\sup\{\hat{f}(X_t)\mathbbm{1}_{[y^\ast,\infty)}(X_t);t\leq T\}\right]=\begin{cases}
g(x), &x\in [y^\ast,\infty),\\
g(y^\ast)\left(\frac{x}{y^\ast}\right)^{\varrho}, &x\in(0,y^\ast).
\end{cases}
$$
It is at this point worth emphasizing that in this jump-diffusion setting verifying optimality by investigating the behavior of the representing function $\hat{f}$ is easier than by investigating the behavior of the generator of the underlying process.

\section{Conclusions}

We considered the representation of the value of a class of optimal stopping problems of linear diffusions as the expected supremum of a function with known
regularity and monotonicity properties. By focusing on the single exercise boundary case, we developed an explicit integral representation for the above mentioned
function by first computing the probability distribution of the running supremum
of the underlying diffusion and then utilizing this distribution in determining the expected value explicitly in terms of the increasing minimal excessive
mapping and the infinitesimal characteristics of the diffusion.

There are at least three directions towards which our analysis could be potentially extended. First, the present approach focuses on single boundary problems and consequently overlooks
general problems with more boundaries. Extending our analysis towards this setting and computing the representing function explicitly would, therefore, constitute a natural extension of our approach.
Second, impulse control and optimal switching problems can in
many diffusion cases be interpreted as sequential stopping problems of the underlying process. Thus, extending our representation to that setting would be interesting too
(for a recent approach to this within impulse control, see \cite{ChSa15}). However, given the potential
discreteness of the optimal policy in the impulse control policy setting seems to make the explicit determination of the integral representation a very
challenging problem which at the moment is outside the scope of our study.\\

\noindent{\bf Acknowledgements:} We would like to thank an {\em anonymous Associate Editor}, {\em Peter Bank}, {\em Paavo Salminen}, and the participants of the workshop {\em Strategic Aspects of Optimal Stopping and Control in Economics and Finance} at {\em ZIF} (Bielefeld University) as well as the participants of the conference {\em Optimization of the flow of dividends: 20 years after} at the Palais Brongniart (Paris) for valuable suggestions and helpful comments.\\

\bibliographystyle{amsplain}
\bibliography{Maks}

\appendix

\section{Proof of Lemma \ref{jointprobabilitydist}}\label{app2}
\begin{proof}
(A) Assume that $h:\mathcal{I}\mapsto \mathbb{R}$ is such that $h\in \mathcal{L}_r^1(\mathcal{I})$. Noticing that
\begin{align*}
\p_x[X_T\in dv|M_T\leq y]&=\frac{\p_x[X_T\in dv;M_T\leq y]}{\p_x[M_T\leq y]}=\frac{\p_x[X_T\in dv;T < \tau_y]}{\p_x[T<\tau_y]}
\end{align*}
demonstrates that
$$
\frac{1}{r}\E_x[h(X_T)|M_T\leq y] = \frac{\E_x\int_0^{\tau_y}e^{-rs}h(X_s)ds}{1-\E_x[e^{-r\tau_y}]}.
$$
Invoking the strong Markov property and utilizing the known form of the Laplace transform of the first hitting time $\tau_y$ (p. 18 on \cite{BorSal15})
yields \eqref{condexpb}. On the other hand,
combining the joint probability density  $\p_x[X_T\in dv, M_T \in dy]$ of $X$ and $M$ stated on p. 26 of \cite{BorSal15}
with the density $\p_x[M_T\in dy]=(\psi'(y)\psi(x)/\psi^2(y))dy$ yields
\begin{align*}
\p_x[X_T\in dv|M_T=y] &=r\frac{S'(y)}{\psi'(y)}\psi(v)m'(v)dv,\quad x\in(a,y].
\end{align*}
The proposed expectation \eqref{condexp} then follows by standard integration. Letting $x\uparrow y$ in \eqref{condexpb}
and invoking L'Hospital's rule yields
$$
\lim_{x\uparrow y}\frac{1}{r}\E_x[h(X_T)|M_T\leq y] = (R_rh)(y)-(R_rh)'(y)\frac{\psi(y)}{\psi'(y)}.
$$
Applying now the representation \eqref{Green} yields the proposed identity \eqref{condexpc}, thus completing the proof of part (A).
(B) Finally, identity \eqref{condexprep} follows under the assumption of the lemma from \eqref{canonical} and \eqref{condexp} after letting $y\downarrow a$ and noticing that $(L_\psi\mathbbm{1})(y)=\psi'(y)/S'(y)$.
\end{proof}

\section{Proof of Lemma \ref{increasedvolatility}}\label{app3}
\begin{proof}
In order to prove the alleged comparative static results, we first denote by $\tilde{\psi}$ the increasing fundamental solution associated with the more volatile dynamics characterized by the volatility coefficient $\tilde{\sigma}(x)\geq\sigma(x)$ for all $x\in \mathcal{I}$. Denote by $\tau_{(y_0,y_1)}=\inf\{t\geq 0:X_t\not\in(y_0,y_1)\}$ the first exit date of the diffusion $X$ from the open subset $(y_0,y_1)\subset\mathcal{I}$, where $a<y_0<y_1 < b$. Standard application of Dynkin's theorem shows that for all $x\in (y_0,y_1)$ we have
\begin{align*}
\E_x\left[e^{-r\tau_{(y_0,y_1)}}\tilde{\psi}(X_{\tau_{(y_0,y_1)}})\right] = \tilde{\psi}(x)+\E_x\int_0^{\tau_{(y_0,y_1)}}e^{-rs}(\mathcal{G}_r\tilde{\psi})(X_s)ds\leq \tilde{\psi}(x),
\end{align*}
since
$$
(\mathcal{G}_r\tilde{\psi})(x) = \frac{1}{2}(\sigma^2(x)-\tilde{\sigma}^2(x))\tilde{\psi}''(x)\leq 0
$$
by the assumed convexity of $\tilde{\psi}$. Consequently, by utilizing standard computations we notice that for all $x\in(y_0,y_1)$ it holds
\begin{align*}
\tilde{\psi}(x) &\geq \tilde{\psi}(y_0)\E_x\left[e^{-r\eta_{y_0}};\eta_{y_0}<\eta_{y_1}\right] + \tilde{\psi}(y_1)\E_x\left[e^{-r\eta_{y_1}};\eta_{y_0}>\eta_{y_1}\right]\\
&\geq \tilde{\psi}(y_1)\E_x\left[e^{-r\eta_{y_1}};\eta_{y_0}>\eta_{y_1}\right]
=\tilde{\psi}(y_1)\frac{\psi(x)-\varphi(x)\frac{\psi(y_0)}{\varphi(y_0)}}{\psi(y_1)-\varphi(y_1)\frac{\psi(y_0)}{\varphi(y_0)}}
\end{align*}
where $\eta_z=\inf\{t\geq 0:X_t=z\}$ denotes the first hitting time of $X$ to a state $z\in \mathcal{I}$. Letting $y_0\downarrow a$ and utilizing the fact that $\lim_{x\downarrow a}\psi(x)/\varphi(x)=0$ for the considered class of boundary behaviors (cf. \cite{BorSal15}, p. 19) then shows that
$$
\frac{\tilde{\psi}(x)}{\tilde{\psi}(y_1)}\geq \frac{\psi(x)}{\psi(y_1)}
$$
for all $x\in(a,y_1)$. On the other hand, noticing that
$$
\frac{\psi(x)}{\psi(y_1)} = \exp\left(-\int_{x}^{y_1}\frac{\psi'(t)}{\psi(t)}dt\right) \leq \exp\left(-\int_{x}^{y_1}\frac{\tilde{\psi}'(t)}{\tilde{\psi}(t)}dt\right) = \frac{\tilde{\psi}(x)}{\tilde{\psi}(y_1)}
$$
for all $a<x\leq y_1 < b$ implies that $\psi'(x)/\psi(x)\geq \tilde{\psi}'(x)/\tilde{\psi}(x)$ for all $x\in\mathcal{I}$, thus completing the proof of our lemma.
\end{proof}

\section{Proof of Lemma \ref{Martin1}}\label{app4}
\begin{proof}
It is clear that under our assumptions $V_{y^\ast}(x)$ is nonnegative, continuous, and  dominates the exercise payoff $g(x)$ for all $x\in \mathcal{I}$.
Moreover, since $y^\ast\in\{x\in \mathcal{I}: V(x)=g(x)\}$ by Theorem 2.1 in \cite{ChIr11}, we find that the stopping region is nonempty.
Let $x_0\in(y^\ast,b)\setminus\mathcal{P}$ be a fixed reference point and define the ratio $h_{x_0}(x)=V_{y^\ast}(x)/V_{y^\ast}(x_0)=V_{y^\ast}(x)/g(x_0)$.
It is clear that our assumptions combined with \eqref{linearityGenerator} guarantee that
$$
\sigma_{x_0}^{h_{x_0}}((x,b])=\frac{\psi(x_0)}{Bg(x_0)}\left[\frac{g'(x+)}{S'(x)}\varphi(x)-g(x)\frac{\varphi'(x)}{S'(x)}\right]=
-\frac{\psi(x_0)}{Bg(x_0)}(L_\varphi g)(x+)
$$
is nonnegative and nonincreasing for all $x\geq x_0$. Analogously,
\begin{align*}
\sigma_{x_0}^{h_{x_0}}([a,x))&=
\frac{\varphi(x_0)}{Bg(x_0)}\left[g(x)\frac{\psi'(x)}{S'(x)}-\frac{g'(x-)}{S'(x)}\psi(x)\right]\mathbbm{1}_{(y^\ast,x_0]}(x) \\
&=
\frac{\varphi(x_0)}{Bg(x_0)}(L_\psi g)(x-)\mathbbm{1}_{(y^\ast,x_0]}(x)
\end{align*}
is nonnegative and nondecreasing for all $x\leq x_0$.  Moreover,
noticing that
$
\sigma_{x_0}^{h_{x_0}}([a,x_0))+\sigma_{x_0}^{h_{x_0}}((x_0,b])=1
$
shows,  by imposing the condition $\sigma_{x_0}^{h_{x_0}}(\{x_0\})=0$, that $\sigma_{x_0}^{h_{x_0}}$ constitutes a probability measure.
Therefore, it induces an $r$-excessive function $h_{x_0}(x)$ via its
Martin representation (cf. Proposition 3.3 in \cite{Salminen1985}). However, since increasing linear transformations of excessive functions are excessive and
$h_{x_0}(x)g(x_0)=V_{y^\ast}(x)$, we observe that $V_{y^\ast}(x)$ constitutes an $r$-excessive majorant of $g$ for $X$. Invoking now \eqref{eq Vy} shows that
$V(x)=V_{y^\ast}(x)$ and consequently, that $\tau_{y^\ast}=\inf\{t\geq 0: X_t\geq y^\ast\}$ is an optimal stopping time.
\end{proof}

\end{document}